\documentclass{article}
\usepackage{graphicx}
\usepackage{amsmath}
\usepackage{amsfonts}
\usepackage{amssymb}
\newtheorem{theorem}{Theorem}

\newtheorem{corollary}[theorem]{Corollary}

\newtheorem{definition}[theorem]{Definition}

\newenvironment{proof}[1][Proof]{\textbf{#1.} }{\ \rule{0.5em}{0.5em}}

\begin{document}

\title{Binary nullity, Euler circuits and interlace polynomials}
\author{Lorenzo Traldi\\Department of Mathematics, Lafayette College\\Easton, PA\ 18042 USA}
\date{}
\maketitle
\begin{abstract}
A theorem of Cohn and Lempel [J. Combin. Theory Ser. A \textbf{13} (1972),
83-89] gives an equality relating the number of circuits in a directed circuit
partition of a 2-in, 2-out digraph to the $GF(2)$-nullity of an associated
matrix. This equality is essentially equivalent to the relationship between
directed circuit partitions of 2-in, 2-out digraphs and vertex-nullity
interlace polynomials of interlace graphs. We present an extension of the
Cohn-Lempel equality that describes arbitrary circuit partitions in
(undirected) 4-regular graphs. The extended equality incorporates topological
results that have been of use in knot theory, and it implies that if $H$ is
obtained from an interlace graph by attaching loops at some vertices then the
vertex-nullity interlace polynomial $q_{N}(H) $ is essentially the generating
function for certain circuit partitions of an associated 4-regular graph.


\end{abstract}

\section{Introduction}

Cohn and Lempel \cite{CL} gave a simple formula relating the number of orbits
in a finite set under a certain kind of permutation to the nullity of an
associated binary matrix. Let $\sigma$ be the cyclic permutation
$\sigma=(1...m)$ of the set $\{1,...,m\}$, let $\sigma_{1},...,\sigma_{k}$ be
pairwise disjoint transpositions of elements of $\{1,...,m\}$, and let
$\pi=\sigma\sigma_{1}...\sigma_{k}$. Let $I_{\pi}$ be the symmetric $k\times
k$ matrix over $GF(2)$ with $(I_{\pi})_{ij}=1$ if and only if $\sigma
_{i}=(ab)$ and $\sigma_{j}=(cd)$ with either $a<c<b<d$ or $c<a<d<b$.

\begin{theorem}
\label{cle}(Cohn-Lempel equality) The number of orbits in $\{1,...,m\}$ under
$\pi=\sigma\sigma_{1}...\sigma_{k}$ is $1+\nu(I_{\pi})$, where $\nu(I_{\pi})$
is the $GF(2)$-nullity of $I_{\pi}$.
\end{theorem}

The Cohn-Lempel equality was reproven by Moran \cite{M} and Stahl \cite{S}. It
was extended to non-disjoint transpositions $\sigma_{1},...,\sigma_{k}$ by
Beck and Moran \cite{Be, BM}, who also pointed out that an equivalent equality
was obtained much earlier by Brahana \cite{Br}. Other related results have
been presented by Macris and Pul\'{e} \cite{MP}, Lauri \cite{Lau} and Jonsson
\cite{Jo}.

\medskip

Suppose $D$ is a connected 2-in, 2-out digraph with $V(D)=\{v_{1},...,v_{n}\}
$ and $E(D)=\{e_{1},...,e_{m}\}$; $D$ may have loops or multiple edges. A
\textit{directed trail} in $D$ is described by a sequence $v_{j_{1}}e_{j_{1}%
}v_{j_{2}}e_{j_{2}}...v_{j_{c}}e_{j_{c}}v_{j_{c+1}}$ of vertices and pairwise
distinct edges such that each $e_{j_{k}}$ is directed from $v_{j_{k}}$ to
$v_{j_{k+1}}$; a trail may also be described by its sequence of edges. If
$v_{j_{c+1}}=v_{j_{1}}$ the trail is a \textit{circuit}; the same circuit is
described if the sequence is permuted cyclically, with the natural notation
changes at the ends. $D$ must have a directed Euler circuit, \textit{i.e.}, a
directed circuit that includes every edge. We presume the edges are indexed so
that $e_{1}...e_{m}$ is an Euler circuit, which we denote $C$. A partition $P$
of $E(G)$ into directed circuits is associated to a permutation $\pi_{P}$ of
$\{1,...,m\}$, with $i\pi_{P}=j$ if $e_{i}$ is followed immediately by $e_{j}$
in one of the directed circuits of $P$. The elements of $P$ correspond to the
orbits in $\{1,...,m\} $ under $\pi_{P}$. $P$ may also be specified by giving
the subset $S_{P}\subseteq V(D)$ consisting of the vertices at which the
incident circuit(s) of $P$ do not follow the same edge-to-edge
\textit{transitions} as $C$ \cite{K}. If the edges directed into such a vertex
$v$ are $e_{a-1}$ and $e_{b-1}$, and those directed outward are $e_{a}$ and
$e_{b}$, then saying that the incident circuit(s) of $P$ do not follow the
same transitions as $C$ means that $C$ is $e_{a-1}e_{a}...e_{b-1}e_{b}...$ and
the incident circuit(s) of $P$ are $e_{a-1}e_{b}...$ and $...e_{b-1}e_{a} $.
(Obvious changes in indexing may be required if any of these edges is a loop
or if $1\in\{a,b\}$.) The permutation $\pi_{P}$ is then of the form
$\sigma\sigma_{1}...\sigma_{k}$, with a transposition $\sigma_{i}$ associated
to each $v\in S_{P}$; if $e_{a}$ and $e_{b}$ are the edges directed outward
from $v$ then $\sigma_{i}$ is $(ab)$.

\medskip

Following \cite{RR}, let $I(D,C)$ be the \textit{interlace matrix} of $D$ with
respect to $C$: the $n\times n$ matrix over $GF(2)$ whose $ij$ entry is 1 if
and only if $i\neq j$ and $v_{i}$ and $v_{j}$ are \textit{interlaced} in $C$,
\textit{i.e.}, when we follow $C$ starting at $v_{i}$ we encounter $v_{j}$,
then $v_{i}$, then $v_{j}$ again before finally returning to $v_{i}$. If $P$
is a directed circuit partition of $D$ then $I_{\pi_{P}}$ is simply the
submatrix of $I(D,C)$ that involves the rows and columns corresponding to
elements of $S_{P}$.

\medskip

\begin{corollary}
\label{clec}Let $D$ be a connected 2-in, 2-out digraph, and let $\mathcal{P}%
(D)$ be the set of partitions of $E(D)$ into directed circuits. Let $I(D,C)$
be the interlace matrix corresponding to an Euler circuit $C$ of $D$, and for
each subset $S\subseteq V(D)$ let $I_{S}(D,C)$ be the submatrix of $I(D,C)$
that involves the rows and columns corresponding to elements of $S$. Then
\[
\sum_{P\in\mathcal{P}(D)}(y-1)^{\left|  P\right|  -1}=\sum_{S\subseteq
V(D)}(y-1)^{\nu(I_{S}(D,C))}.
\]
\end{corollary}

\begin{proof}
$P\leftrightarrow S_{P}$ defines a one-to-one correspondence between elements
of $\mathcal{P}(D)$ and subsets of $V(D)$, and the Cohn-Lempel equality tells
us that for each $P\in\mathcal{P}(D)$, $\left\vert P\right\vert =\nu
(I_{\pi_{P}})+1=\nu(I_{S_{P}}(D,C))+1$.
\end{proof}

\medskip

Arratia, Bollob\'{a}s, and Sorkin introduced the \textit{interlace
polynomials} of looped, undirected graphs in \cite{A1, A2, A}. These
invariants were first defined recursively, but soon it was shown that they are
also given by formulas involving matrix nullities \cite{AH, A}. Given an
undirected graph $G$ with $V(G)=\{v_{1},...,v_{n}\}$, let $\mathcal{A}(G)$ be
the $n\times n$ matrix with entries in $GF(2)$ given by $a_{ii}=1$ if and only
if $v_{i}$ is looped, and for $i\neq j$, $a_{ij}=1$ if and only if $v_{i}$ and
$v_{j}$ are adjacent. For $S\subseteq V(G)$ let $\mathcal{A}(G)_{S}$ denote
the submatrix of $\mathcal{A}(G)$ consisting of the rows and columns
corresponding to elements of $S$; equivalently $\mathcal{A}(G)_{S}%
=\mathcal{A}(G[S])$, where $G[S]$ denotes the subgraph of $G$ induced by $S$ .

\medskip

\begin{definition}
\label{int}The vertex-nullity interlace polynomial of $G$ is
\[
q_{N}(G)=\sum_{S\subseteq V(G)}(y-1)^{\nu(\mathcal{A}(G)_{S})}%
\]
and the (two-variable) interlace polynomial of $G$ is
\[
q(G)=\sum_{S\subseteq V(G)}(x-1)^{\left\vert S\right\vert -\nu(\mathcal{A}%
(G)_{S})}(y-1)^{\nu(\mathcal{A}(G)_{S})}.
\]
\end{definition}

Definition \ref{int} may be applied to graphs with parallel edges or parallel
loops, but parallels do not affect $\mathcal{A}(G)$ or the interlace polynomials.

\medskip

Suppose $D$ is a connected 2-in, 2-out digraph with an Euler circuit $C$, and
$H$ is the \textit{interlace graph} of $D$ with respect to $C$, \textit{i.e.},
the undirected graph with $V(H)=V(D)$ and $\mathcal{A}(H)=I(D,C)$. Theorem 24
of \cite{A2} states that $q_{N}(H)$ is essentially the same as the generating
function for partitions of $E(D)$ into directed circuits. The proof given
there involves the recursive definition of $q_{N}$, but once it is recognized
that $q_{N}$ can also be given by Definition \ref{int}, it becomes clear that
the relationship between $q_{N}(H)$ and directed circuit partitions of $D$ is
equivalent to Corollary \ref{clec} above.

\medskip

The Kauffman bracket polynomial of a knot or link diagram (and other link
invariants too) can be given by a sum whose terms are obtained by counting
circuits in circuit partitions. As Arratia, Bollob\'{a}s, and Sorkin observed
in \cite{A2}, this leads directly to a relationship between the Kauffman
bracket and the vertex-nullity interlace polynomial. The fact that the
Kauffman bracket can be described by formulas involving $GF(2)$-nullity has
also been noted by knot theorists; Soboleva \cite{So} seems to have been the
first to explicitly cite the Cohn-Lempel equality. Some of the formulas used
by knot theorists resemble the Cohn-Lempel equality or Corollary \ref{clec}
without being quite the same. For instance, Zulli \cite{Z} counted circuits
using a formula that involves the $GF(2)$-nullities of matrices that may have
nonzero entries on the diagonal, and are all $n\times n$. More recently, Lando
\cite{L} and Mellor \cite{Me} used a formula that includes both the
Cohn-Lempel equality and Zulli's formula. As is natural in the literature of
knot theory, the discussions in these references are essentially topological
-- the arguments of Lando and Zulli involve the homology of surfaces, and
Mellor and Soboleva are concerned with weight systems for link invariants --
and they focus (implicitly or explicitly)\ on connected, planar digraphs.

\medskip

In this note we present a combinatorial proof of an extended version of the
Cohn-Lempel equality that applies to arbitrary circuit partitions in arbitrary
4-regular graphs. This extended Cohn-Lempel equality does not require that the
4-regular graph in question be connected, directed or planar, and it includes
the various formulas just mentioned. The greater generality of the extended
Cohn-Lempel equality is not only pleasing but also useful: it is a crucial
part of an interlacement-based analysis of Kauffman's bracket for virtual
links \cite{Kv} developed by Zulli and the present author \cite{T, TZ}, and as
we see below it allows us to extend the relationship between circuit
partitions and interlace polynomials to include interlace graphs that have had
some loops attached.

\medskip

Before stating the extended Cohn-Lempel equality we take a moment to establish
notation and terminology. Suppose $G$ is an undirected 4-regular graph. If $G$
is connected it must have an Euler circuit $C$. Choose one of the two
orientations of $C$, let $D$ be the 2-in, 2-out digraph obtained from $G$ by
directing all edges according to that orientation, and let $I(D,C)$ be the
interlace matrix of $D$ with respect to $C$. If $G$ is not connected then let
$C$ be a set of Euler circuits, one in each of the $c(G)$ connected components
of $G$, and let $D$ be a 2-in, 2-out digraph resulting from one of the
$2^{c(G)}$ possible choices of orientations for the circuits in $C$. The
interlace matrix $I(D,C)$ then consists of $c(G)$ diagonal blocks
corresponding to the interlace matrices of the components of $G$ with respect
to the circuits of $C$; the entries outside these diagonal blocks are all 0.

\medskip

Let $P$ be a partition of $E(G)$ into undirected circuits. Suppose $v_{i}\in
V(G)$, and consider an edge $e$ that is directed toward $v_{i}$ in $D$. Some
circuit of $P$ must contain $e$. If we follow this circuit through $v_{i}$
after traversing $e$, then there are three ways we might leave $v_{i}$: along
the edge $C$ uses to leave $v_{i}$ after arriving along $e$, along the other
edge directed away from $v_{i}$ in $D$, or along the remaining edge directed
toward $v_{i}$ in $D$. We say $P$ \textit{follows }$C$ \textit{through}
$v_{i}$ in the first case, $P$ is \textit{orientation-consistent at} $v_{i}%
$\textit{\ but does not follow} $C$ in the second case, and $P$ is
\textit{orientation-inconsistent at} $v_{i}$ in the third case. Changing the
choice of $e$ or the orientations of the circuit(s) in $C$ does not affect the
descriptions of the three cases. (\textit{N.B}. In order to provide
well-defined descriptions of the three possibilities at looped vertices we
should actually refer to \textit{half-edges}; we leave this sharpening of
terminology to the reader.) A matrix $I_{P}=I_{P}(D,C)$ is obtained from
$I(D,C)$ as follows. If $P$ follows $C$ through $v_{i}$ then the row and
column of $I(D,C)$ corresponding to $v_{i}$ are removed; if $P$ is
orientation-consistent at $v_{i}$ but does not follow $C$ then the row and
column of $I(D,C)$ corresponding to $v_{i}$ are retained without change; and
if $P$ is orientation-inconsistent at $v_{i}$ then the row and column of
$I(D,C)$ corresponding to $v_{i}$ are retained with one change: their common
diagonal entry is changed from $0$ to $1$.

\begin{theorem}
\label{ecle}(Extended Cohn-Lempel equality) If $G$ is an undirected, 4-regular
graph with $c(G)$ components and $P$ is a partition of $E(G)$ into undirected
circuits then
\[
\left|  P\right|  =\nu(I_{P})+c(G).
\]
\end{theorem}

As an example of the extended Cohn-Lempel equality consider the complete graph
$K_{5}$, with vertices denoted 1, 2, 3, 4 and 5. Let $D$ be the directed
version of $K_{5}$ with edge-directions given by the Euler circuit $C=$
1234513524. If $P$ follows $C$ at vertex 1, is orientation-inconsistent at
vertices 2 and 3, and is orientation-consistent but does not follow $C$ at
vertices 4 and 5 then
\[
\nu(I_{P}(D,C))=\nu\left(
\begin{array}
[c]{cccc}%
1 & 0 & 1 & 0\\
0 & 1 & 1 & 1\\
1 & 1 & 0 & 0\\
0 & 1 & 0 & 0
\end{array}
\right)  =0,
\]
so $\left|  P\right|  =1$. The one circuit in $P$ is the Euler circuit
1254231534. The partition $P^{\prime}$ that disagrees with $P$ only by
following $C$ at 3 corresponds to the matrix $I_{P^{\prime}}(D,C)$ obtained by
removing the second row and column of $I_{P}(D,C)$, so $\nu(I_{P^{\prime}%
}(D,C))=1$. $P^{\prime}$ contains the circuits 1254234 and 135.

\medskip

The extended Cohn-Lempel equality implies that the relationship between
interlace polynomials and directed circuit partitions extends to looped
interlace graphs.

\begin{corollary}
\label{interlace}Suppose $C$, $D$ and $G$ are as in Theorem \ref{ecle}, and
$H$ is obtained from the interlace graph of $D$ with respect to $C$ by
attaching loops at some vertices. Then
\[
q_{N}(H)=\sum_{S\subseteq V(H)}(y-1)^{\left\vert P_{S}\right\vert -c(G)},
\]
where $P_{S}$ is the undirected circuit partition that follows $C$ at each
vertex $v\not \in S$, is orientation-inconsistent at each looped vertex $v\in
S$, and is orientation-consistent but does not follow $C$ at each unlooped
vertex $v\in S$. Also, the two-variable interlace polynomial of $H$ is
\[
q(H)=\sum_{S\subseteq V(H)}(x-1)^{\left\vert S\right\vert -\left\vert
P_{S}\right\vert +c(G)}(y-1)^{\left\vert P_{S}\right\vert -c(G)}.
\]
\end{corollary}

In the balance of the paper we prove Theorem \ref{ecle}, derive an analogue of
Corollary \ref{interlace} for the multivariate interlace polynomial of
Courcelle \cite{C}, and comment briefly on related results of Beck and Moran
\cite{Be, BM}, Macris and Pul\'{e} \cite{MP}, Lauri \cite{Lau} and Jonsson
\cite{Jo}. Before proceeding we should express our gratitude to D. P. Ilyutko
and L. Zulli, whose discussions of \cite{IM} and \cite{Z} inspired this note.
We are also grateful to Lafayette College for its support.

\section{Proof of the extended Cohn-Lempel equality}

The equality is proven under the assumption that $G$ is connected; the general
case follows as the contributions from different connected components are
simply added together.

\medskip

We begin with a special case: every entry of $I_{P}$ is 0. This case falls
under the original Cohn-Lempel equality but we provide an argument anyway, for
the sake of completeness. If $I_{P}$ is the empty matrix, then $P=\{C\}$ and
the equality is satisfied. If $I_{P}$ is the $1\times1$ matrix $(0)$ and the
one entry corresponds to $a$, let $aC_{1}a$ and $aC_{2}a$ be circuits with
$C=aC_{1}aC_{2}a$; then $P$ consists of two separate circuits $aC_{1}a$ and
$aC_{2}a$, so the equality is satisfied. Proceeding by induction on the size
of $I_{P}=\mathbf{0}$, let $S_{P}$ be the set of vertices at which $P$ does
not follow $C$. Choose $a\in S_{P}$ so that $C=aC_{1}aC_{2}a$ with $C_{1}$ as
short as possible. Then no element of $S_{P} $ appears on $C_{1}$, for a
vertex that appears only once is interlaced with $a$ (violating $I_{P}%
=\mathbf{0}$) and a vertex $b$ that appears twice has $C=bC_{1}^{\prime}%
bC_{2}^{\prime}b$ with $C_{1}^{\prime} $ shorter than $C_{1}$ (violating the
choice of $a$). Let $Q$ be the circuit partition that disagrees with $P$ only
by following $C$ at $a$. Then $I_{Q}$ is smaller than $I_{P}$, so the
inductive hypothesis tells us that $\left\vert Q\right\vert =\nu(I_{Q}%
)+1=\nu(I_{P})$. As $C=aC_{1}aC_{2}a$ and both $P$ and $Q$ follow $C$ at every
vertex of $C_{1}$, it is clear that $a$ appears on two circuits of $P$
($aC_{1}a$ and another), these two circuits are united in $Q$, and the other
elements of $P$ and $Q$ coincide. Hence $\left\vert P\right\vert =\left\vert
Q\right\vert +1=\nu(I_{P})+1$.

\medskip

If $I_{P}$ is the $1\times1$ matrix $(1)$ with a single entry corresponding to
$a$, and $C=aC_{1}aC_{2}a$, then the equality is satisfied because $P$
contains only the Euler circuit $aC_{1}a\bar{C}_{2}a$. Here $\bar{C}_{2}$ is
the reverse of $C_{2}$ and the Euler circuit $aC_{1}a\bar{C}_{2}a$ is the
$\kappa$\textit{-transform} of $C$ at $a$, denoted $C\ast a$ \cite{Bc, K}.

\medskip

The argument proceeds by induction on the size of $I_{P}\neq\mathbf{0}$.
Suppose $P$ is orientation-inconsistent with $C$ at a vertex $a$, and let
$C=aC_{1}aC_{2}a$. Then $C\ast a=aC_{1}a\bar{C}_{2}a$ is also an Euler circuit
of $G$, and $P$ follows $C\ast a$ through $a$. If $v\neq a$ is a vertex that
appears on both $C_{1}$ and $C_{2}$ then either $P$ follows both $C$ and
$C\ast a$ through $v$, or else $P$ is orientation-inconsistent with respect to
one of $C,C\ast a$ at $v$ and orientation-consistent with the other of
$C,C\ast a$ without following it through $v$. If $v\neq a$ is a vertex that
appears on only one of $C_{1},C_{2}$ then $P$ has the same status with respect
to $C$ and $C\ast a$ at $v$. If $a\not \in\{v,w\}$ and $v $ and $w$ both
appear on $C_{1}$ and $C_{2}$, then the interlacement of $v$ and $w$ with
respect to $C\ast a$ is the opposite of their interlacement with respect to
$C$. On the other hand, if $a\not \in\{v,w\}$ and either $v$ or $w$ doesn't
appear on both $C_{1}$ and $C_{2}$ then their interlacement with respect to
$C\ast a$ is the same as their interlacement with respect to $C$. In sum, if
$D\ast a$ denotes the digraph on $G$ consistent with $C\ast a $ then
\[
I_{P}(D,C)=\left(
\begin{array}
[c]{ccc}%
1 & \mathbf{1} & \mathbf{0}\\
\mathbf{1} & M_{11} & M_{12}\\
\mathbf{0} & M_{21} & M_{22}%
\end{array}
\right)  \qquad\mathrm{and}\qquad I_{P}(D\ast a,C\ast a)=\left(
\begin{array}
[c]{cc}%
\bar{M}_{11} & M_{12}\\
M_{21} & M_{22}%
\end{array}
\right)
\]
for appropriate submatrices $M_{ij}$; here $\bar{M}_{11}$ differs from
$M_{11}$ in every entry. Adding the first row of $I_{P}(D,C)$ to each row
involved in $M_{11}$ and $M_{12}$, we see that $I_{P}(D\ast a,C\ast a)$ and
$I_{P}(D,C)$ have the same nullity. As $I_{P}(D\ast a,C\ast a)$ is smaller
than $I_{P}(D,C)$, induction tells us that $\left\vert P\right\vert =\nu
(I_{P}(D,C))+1$.

\medskip

Suppose now that there is no vertex at which $P$ is orientation-inconsistent;
this case too falls under the original Cohn-Lempel equality. As the equality
has already been verified in the case $I_{P}=\mathbf{0}$, we presume that
there are two interlaced vertices $a$ and $b$ such that $P$ follows $C$
neither at $a$ nor at $b$. Let $C=aC_{1}bC_{2}aC_{3}bC_{4}a$, and $C\ast a\ast
b\ast a=aC_{1}bC_{4}aC_{3}bC_{2}a$; then $P$ follows $C\ast a\ast b\ast a$
through both $a$ and $b$. A case-by-case analysis shows that
\begin{gather*}
I_{P}(D,C)=\left(
\begin{array}
[c]{cccccc}%
\mathbf{0} & \mathbf{1} & \mathbf{1} & \mathbf{1} & \mathbf{0} & \mathbf{0}\\
\mathbf{1} & \mathbf{0} & \mathbf{1} & \mathbf{0} & \mathbf{1} & \mathbf{0}\\
\mathbf{1} & \mathbf{1} & M_{11} & M_{12} & M_{13} & M_{14}\\
\mathbf{1} & \mathbf{0} & M_{21} & M_{22} & M_{23} & M_{24}\\
\mathbf{0} & \mathbf{1} & M_{31} & M_{32} & M_{33} & M_{34}\\
\mathbf{0} & \mathbf{0} & M_{41} & M_{42} & M_{43} & M_{44}%
\end{array}
\right) \\
~\\
\mathrm{and}\qquad I_{P}(D,C\ast a\ast b\ast a)=\left(
\begin{array}
[c]{cccc}%
M_{11} & \bar{M}_{12} & \bar{M}_{13} & M_{14}\\
\bar{M}_{21} & M_{22} & \bar{M}_{23} & M_{24}\\
\bar{M}_{31} & \bar{M}_{32} & M_{33} & M_{34}\\
M_{41} & M_{42} & M_{43} & M_{44}%
\end{array}
\right)
\end{gather*}
for appropriate submatrices. For instance, suppose $v_{i}$ appears in $C_{2}$
and $C_{3}$ and $v_{j}$ appears in $C_{2}$ and $C_{4}$; then $v_{i}$ is
interlaced with $v_{j}$ with respect to $C$ if and only if $v_{i}$ precedes
$v_{j}$ in $C_{2}$, whereas $v_{i}$ is interlaced with $v_{j}$ with respect to
$C\ast a\ast b\ast a$ if and only if $v_{i}$ follows $v_{j}$ in $C_{2}$.
Consequently the entry of $I_{P}(D,C)$ corresponding to $v_{i}$ and $v_{j}$,
which falls in $M_{13}$ and $M_{31}$, is the opposite of the corresponding
entry of $I_{P}(D,C\ast a\ast b\ast a)$. Using elementary row operations we
see that $I_{P}(D,C\ast a\ast b\ast a)$ and $I_{P}(D,C)$ have the same
nullity; as $I_{P}(D,C\ast a\ast b\ast a)$ is smaller the inductive hypothesis
tells us that $\left\vert P\right\vert =\nu(I_{P}(D,C))+1$.

\medskip

Readers familiar with \cite{AH, A} will recognize the matrix reductions in the
argument. The same reductions are used to deduce the recursive properties of
the interlace polynomials from their matrix formulas.

\section{The multivariate interlace polynomial}

Courcelle \cite{C} introduced a multivariate interlace polynomial of a looped
graph $H$, given by
\[
C(H)=\sum_{\substack{A,B\subseteq V(H) \\A\cap B=\emptyset}}\left(
\prod_{a\in A}x_{a}\right)  \left(  \prod_{b\in B}y_{b}\right)  u^{\left\vert
A\cup B\right\vert -\nu((H\nabla B)[A\cup B])}v^{\nu((H\nabla B)[A\cup B])}%
\]
where $H\nabla B$ denotes the graph obtained from $H$ by toggling loops at the
vertices in $B$ and $u$, $v$, and the various $x_{a}$ and $y_{b}$ are
independent indeterminates. The contribution of each $A,B$ pair to $C(H)$ is
distinguished by the corresponding indeterminates. Consequently if $D$ is a
2-in, 2-out digraph and $H$ is a looped version of the interlace graph of $D$
with respect to a set $C$ of directed Euler circuits for the components of $D
$, then the extended Cohn-Lempel equality tells us that $C(H)$ is essentially
the same as the list of all partitions of $E(D)$ into undirected circuits,
with each partition listed along with its cardinality. That is, $C(H)$ is
essentially the \textit{transition polynomial} studied by Jaeger \cite{J} and
Ellis-Monaghan and Sarmiento \cite{E}.

\begin{corollary}
Suppose $D$ is a 2-in, 2-out digraph, $C$ contains a directed Euler circuit
for each component of $D$, and $H$ is obtained from the interlace graph of $D
$ with respect to $C$ by attaching loops at some vertices. Then the
multivariate interlace polynomial of $H$ is
\[
C(H)=\sum_{\substack{A,B\subseteq V(H) \\A\cap B=\emptyset}}\left(
\prod_{a\in A}x_{a}u\right)  \left(  \prod_{b\in B}y_{b}u\right)  (\frac{v}%
{u})^{\left\vert P_{A,B}\right\vert -c(D)}%
\]
where $P_{A,B}$ is the undirected circuit partition that follows $C$ at
vertices not in $A\cup B$, is orientation-inconsistent at looped vertices in
$A$ and unlooped vertices in $B$, and is orientation-consistent but does not
follow $C$ at unlooped vertices in $A$ and looped vertices in $B$.
\end{corollary}

\begin{proof}
Reformulating the definition,
\[
C(H)=\sum_{\substack{A,B\subseteq V(H) \\A\cap B=\emptyset}}\left(
\prod_{a\in A}x_{a}u\right)  \left(  \prod_{b\in B}y_{b}u\right)  (\frac{v}%
{u})^{\nu((H\nabla B)[A\cup B])}.
\]
\end{proof}

\section{Two remarks}

1. The original form of the Cohn-Lempel equality is not completely general.
For instance, if $n\geq3$ then the identity permutation is not $\sigma
\sigma_{1}...\sigma_{k}$ for any disjoint transpositions $\sigma
_{1},...,\sigma_{k}$. Beck and Moran \cite{Be, BM} extended the Cohn-Lempel
equality to arbitrary permutations by removing the requirement that the
$\sigma_{i}$ be disjoint. Theorem \ref{ecle} may also be applied to arbitrary
permutations. If $\pi$ is a permutation of $\{1,...,2n\}$ choose any partition
of $\{1,...,2n\}$ into pairs, and construct the directed graph $D$ whose $n$
vertices correspond to these pairs and whose $2n$ edges correspond to
$1,...,2n$, with the edge corresponding to $i$ directed from the vertex
corresponding to the pair containing $i$ to the vertex corresponding to the
pair containing $i\pi$. If $\pi^{\prime}$ is a permutation of $\{1,...,2n-1\}$%
, first replace it with the permutation $\pi$ of $\{1,...,2n\}$ that has
$i\pi=i\pi^{\prime}$ for $i<2n-1$, $(2n-1)\pi=2n$, and $(2n)\pi=(2n-1)\pi
^{\prime}$; then construct $D$ as before.

\bigskip

2. Macris and Pul\'{e} \cite{MP}, Lauri \cite{Lau} and Jonsson \cite{Jo}
introduced skew-symmetric integer matrices that reduce (mod 2) to $I_{\pi}$
and whose nullity over the rationals is $\nu(I_{\pi})$. In general, however,
there is no skew-symmetric version of $I_{P}(D,C)$ whose $\mathbb{Q} $-nullity
can be used in Theorem \ref{ecle}. For example, consider the directed graph
$D$ with vertices denoted $1,2,3$ (mod 3) in which there are two edges
directed from vertex $i$ to vertex $i+1$ for each $i$. $E(D)$ has a partition
$P$ containing three undirected circuits; each element of $P$ consists of two
parallel edges. $I_{P}(D,C)$ is the $3\times3$ binary matrix with every entry
$1$, and $\nu(I_{P}(D,C))=2$ in accordance with the extended Cohn-Lempel
equality. However a skew-symmetric $3\times3$ matrix of $\mathbb{Q}$-nullity 2
must have at least five of its nine entries equal to 0.

\section{Dedication}

T. H. Brylawski's work has influenced a generation of researchers studying
matroids and the Tutte polynomial. My training in knot theory was focused on
algebraic topology rather than combinatorics, so I particularly appreciated
the clarity and thoroughness of his expository writing. No less important was
his professional hospitality, which made me feel welcome in a new field. This
note is gratefully dedicated to his memory.

\end{document}